\documentclass[12pt]{article}
\usepackage[left=3cm, right=3cm, top=2cm]{geometry}
\usepackage{amsmath, amsthm, amssymb, mathtools}
\allowdisplaybreaks
\usepackage{amscd,amssymb}
\usepackage{comment}
\usepackage{footmisc}
\usepackage[all]{xy}
\usepackage{color}
\usepackage{verbatim}
\usepackage{mathrsfs}
\usepackage{graphicx,epsfig}
\usepackage[T1]{fontenc}
\usepackage[ansinew]{inputenc}
\usepackage{float}
\usepackage{setspace}
\usepackage{xpatch}
\usepackage{amsfonts}

\usepackage[colorlinks=true,linkcolor=blue,urlcolor  = blue]{hyperref}
\usepackage{subcaption}
\usepackage{pst-node}
\usepackage{tikz-cd}
\usetikzlibrary{positioning}
\usetikzlibrary{shapes,arrows}
\usepackage[new]{old-arrows}
\usepackage{tikz}
\usepackage{graphicx,epsfig}
\usepackage{mathrsfs}
\usepackage{verbatim}
\usepackage{comment}
\usepackage{cleveref}
\usepackage{hyperref}
\usepackage{url}
\usepackage{footmisc}

 \usetikzlibrary{calc,patterns,
 	decorations.pathmorphing,
 	decorations.markings, shapes.geometric, graphs, graphs.standard, quotes,shapes,chains,scopes,positioning,arrows}
\usepackage{tikzscale}
\tikzset{|/.tip={Bar[width=.8ex,round]}}
\newtheorem{thm}{Theorem}
\newtheorem{cor}[thm]{Corollary}
\newtheorem{lem}[thm]{Lemma}
\numberwithin{thm}{section}

\usepackage{tikz-cd}
 \newtheorem{prop}[thm]{Proposition}
\theoremstyle{definition}
 
 \newtheorem{rmk}{Remark}

 \newtheorem{ques}{Question}
 \newtheorem*{note}{Note}
 
\newcommand{\A}{\mathcal{A}}
\newcommand{\B}{\mathcal B}

\newcommand{\bS}{\mathbb{S}}
\newcommand{\F}{\mathcal{F}}

\def\F{\mathcal{F}}

\def\dl{\Delta L}
\def\K{\mathcal{K}}

\def\x{\times}

\title{Topology of Clique Complexes of Line Graphs}

\author{Shuchita Goyal\footnote{Chennai Mathematical Institute, Chennai, India. Email: shuchitagoyal23@gmail.com}, Samir Shukla\footnote{Indian Institute of Technology Bombay, Mumbai, India. Email: samirshukla43@gmail.com}, Anurag Singh\footnote{Chennai Mathematical Institute, Chennai, India. Email: anuragsingh@cmi.ac.in}}

\begin{document}

\date{}

\maketitle
\begin{abstract}
The clique complex of a graph $G$ is a simplicial complex whose simplices are all the cliques of $G$, and the line graph $L(G)$ of $G$ is a graph whose vertices are the edges of $G$ and the edges of $L(G)$ are incident edges of $G$.
In this article, we determine the homotopy type of the clique complexes of line graphs for several classes of graphs including  triangle-free graphs, chordal graphs, complete multipartite graphs, wheel-free graphs, and $4$-regular circulant graphs.  We also give a closed form formula for the homotopy type of these complexes in several cases.

\end{abstract}
\medskip


\vspace{0.1cm}


\vspace{.1in}

\hrule

\section{Introduction}
 

For a fixed graph $H$, a subgraph $K$ of a graph $G$ is called $H$-free if $K$ does not contain any subgraph isomorphic to $H$. Finding $H$-free subgraphs of a graph $G$ is a well studied notion in extremal graph theory. It is easy to observe that $H$-freeness is a hereditary property, thereby giving a simplicial complex for any fixed graph $G$, denoted $\F(G,H)$. These complexes have already been studied for different graphs $H$, for example see \cite{barmak, DE08, KK10, meshulam, SM2018, woodroofe} when $H$ is a complete graph, \cite{jon08, anu20, wach03} when $H$ is a star graph, \cite{HL2020, LSW} when $H$ is $r$ disjoint copies of complete graphs on $2$ vertices, denoted $rK_2$. Replacing $H$ by a class of trees on fixed number of vertices has also gained attention in the last decade, for instance see \cite{PSA, PS18, tardos}.


In this article, our focus will be on the complex $\F(G,2K_2)$.  In \cite{LSW},  Linusson et. al have studied $\F(G,rK_2)$ for complete graphs and complete bipartite graphs. They showed that  $\F(G,rK_2)$ is homotopy equivalent to a wedge of $(3r-4)$-dimensional spheres when $G$ is a complete graph, and  $(2r-3)$-dimensional spheres if $G$ is a complete bipartite graph. Further, Holmsen  and Lee \cite{HL2020} studied these complexes for general graph $G$ and they showed that $\F(G,rK_2)$ is $(3r-3)$-Leray\footnote{A simplicial complex $\mathsf{K}$ is called  {\em $d$-Leray} (over a field $\mathbb{F}$) if $\tilde{H}_i(\mathsf{L}, \mathbb{F}) = 0 $ for all $i\geq d$ and  for every induced subcomplex $\mathsf{L}\subseteq \mathsf{K}$.}, and it is  $(2r-2)$-Leray if  $G$ is bipartite. It is worth noting that $\F(G,2K_2)$ is the clique complex of the line graph of $G$, denoted $\Delta L(G)$. Clique complexes of graphs have a very rich literature in topological combinatorics, for instance see \cite{MA13, MK1}. Recently, Nikseresht \cite{AN2020} studied some algebraic properties like Cohen-Macaulay, sequentially Cohen-Macaulay, Gorenstein, etc., of clique complexes of line graphs. In this article, we determine the exact homotopy type of $\dl(\underline{\ \ })$ for various classes of graphs including  triangle-free graphs, chordal graphs, complete multipartite graphs, wheel-free graphs and $4$-regular circulant graphs. Moreover, we show that $\dl(\underline{\ \ })$ in each of these cases is homotopy equivalent to a wedge of equidimensional spheres, except for the $4$-regular circulant graphs for which it is homotopy equivalent to a wedge of circles and 2-spheres. 


This article is organized as follows. In \Cref{sec:prelim}, we recall various definitions and results which are needed. In the next section, we analyze $\dl(G)$ as the $2$-skeleton of the clique complex of $G$ and then use it to compute the homotopy type of $\dl(G)$ for various classes of graphs including triangle free graphs and chordal graphs. In  \Cref{sec:chordal}, by realizing $\dl$ as a functor from the category of graphs to the category of simplicial complexes, we compute the  homotopy type of $\dl(G)$ when $G$ is a complete multipartite graph.
\Cref{sec:wheel} is devoted towards the study of $\dl$ of wheel-free graphs and 4-regular circulant graphs. 
In the last section, we discuss a few questions that arise naturally from the work done in this article. 


\section{Preliminaries}\label{sec:prelim}

A  ({\it simple}) {\it graph} is an ordered pair $G=(V(G),E(G))$, where $V(G)$ is called the set of vertices and $E(G) \subseteq \binom{V(G)}{2}$, the set of (unordered) edges of $G$. The vertices $v_1, v_2 \in V(G)$ are said to be adjacent, if $(v_1,v_2)\in E(G)$ and this  is also denoted by $v_1 \sim v_2$. For $v \in V(G)$, the set of its {\it neighbours}, $N_G(v)$, in $G$ is $\{x \in V(G) : x \sim v\}$. The {\it degree} of a vertex $v\in V(G)$ is the cardinality of the set of its neighbours. A vertex $v$ of degree $1$ is called a {\it leaf vertex} and the vertex adjacent to $v$ is called the {\it parent} of $v$. An edge adjacent to a leaf vertex  is called a  {\it leaf/hanging edge}. A graph $H$ with $V(H) \subseteq V(G)$ and $E(H) \subseteq E(G)$ is called a {\it subgraph} of the graph $G$. 
For a nonempty subset $U \subseteq V(G)$, the induced subgraph $G[U]$, is the subgraph of $G$ with vertices $V(G[U]) = U$ and $E(G[U]) = \{(a, b) \in E(G) : a, b \in U\}$. Similarly, for a nonempty subset $F \subseteq E(G)$, the induced subgraph $G[F]$, is the subgraph of $G$ with vertices $V(G[F]) = \{v \in V(G) : v \in e \ \text{for some} \ e \in F\}$ and $E(G[F]) = F$. In this article, the graph $G[V(G)\setminus A]$ will be denoted by $G-A$ for $A\subsetneq V(G)$.

For $n \geq 1$, the {\it complete graph} on $n$ vertices is a graph where any two distinct vertices are adjacent, and  is denoted by $K_n$. For $n \geq 3$, the {\it cycle graph $C_n$} is the graph with $V(C_n) = \{1, \ldots, n\}$ and $E(C_n) = \{(i, i+1) : 1 \leq i \leq n-1\} \cup \{(1, n)\}$. For $r \geq 1$, the {\it path graph} of length $r$ is a graph with  vertex set $V(P_r) = \{0, \ldots, r\}$ and  edge set $E(P_r) = \{(i,i+1) :  0 \leq i \leq r-1\}$, and is denoted by $P_r$. The {\it line graph}, $L(G)$, of a graph $G$ is the graph with
$V(L(G)) = E(G)$ and $E(L(G)) = \{((v_1, v_2), (v_1^{\prime}, v_2')) : |\{v_1, v_2\} \cap \{v_1', v_2'\}| = 1\}$.

An ({\it abstract}) {\it simplicial complex} $\K$ is a collection of finite sets such that if $\tau \in \K$ and $\sigma \subset \tau$, then $\sigma \in \K$. The elements  of $\K$ are called the  {\it simplices} (or {\it faces}) of $\K$.  If $\sigma \in \K$ and $|\sigma |=k+1$, then $\sigma$ is said to be {\it $k$-dimensional}. The maximal simplices of $\K$ are also called {\it facets} of $\K$. The set of $0$-dimensional simplices of $\K$ is denoted by $V(\K)$, and its elements are called {\it vertices} of $\K$. A {\it subcomplex} of a simplicial complex $\K$ is a simplicial complex whose simplices are contained in $\K$. For $k \geq 0$, the {\it $k$-skeleton} of a simplicial complex $\K$ is a subcomplex consisting of all the simplices of dimension $\leq k$ and it is denoted by $\K^{(k)}$. The {\it cone} of a simplicial complex $\K$ with apex $w$, denoted $\mathcal{C}_w(\K)$, is a simplicial complex whose facets are $\sigma \cup \{w\}$ for each facet $\sigma $ of $\K$. In this article, we always assume empty set as a simplex of any simplicial complex and we consider any simplicial complex as a topological space, namely its  geometric realization. For the definition of geometric realization, we refer to Kozlov's book \cite{dk}.



The {\it clique complex}, $\Delta(G)$, of a graph $G$ is the simplicial complex whose simplices are subsets $\sigma \subseteq  V(G)$ such that 
$G[\sigma]$ is a complete graph. 

The clique complex of line graph of a graph $G$, $\dl(G)$, has also been studied by Linusson, Shareshian, and Welker in \cite{LSW} and  Holmsen  and Lee in \cite{HL2020}. They denoted these complexes by $NM_2(G)$ and proved the following results. 


\begin{thm}[Theorem $1.1$, \cite{LSW}] \label{thm:completegraph}
    Let $n$ be a positive integers. Then $NM_2(K_n)$ is homotopy equivalent to a wedge of spheres of dimension $2$.
\end{thm}

\begin{thm} [Theorem $1.1$, \cite{HL2020}] \label{thm:leray}
    For a graph $G$, the complex $NM_2(G)$ is $3$-Leray. 
\end{thm}


It is worth mentioning here that the Lerayness of the clique complex of the complement of a claw-free graph has also been studied in \cite{BC18} and \cite{Nevo}.

\subsection{(Homotopy) Pushout}
	
   Let $X,Y,Z$ be topological spaces, and $p: X \to Y$ and $q: X \to Z$ be continuous maps. The  {\it pushout} of the diagram $Y \xleftarrow{p}  X \xrightarrow{q} Z$ is the space $$\Big(Y \bigsqcup   Z \Big)/ \sim,$$ where $\sim$ denotes  the equivalence relation $p(x) \sim q(x)$, for $x \in X$.

   The {\it homotopy  pushout} of  $Y \xleftarrow{p}  X \xrightarrow{q} Z$ 
  is the space $\big(Y \sqcup ( X \times [0,1] ) \sqcup  Z\big)/ \sim$, where 
  $\sim$ denotes  the equivalence relation $(x,0) \sim p(x)$, and $(x,1) \sim q(x)$, for $x \in X$. It can be shown that homotopy pushouts of any two homotopy equivalent diagrams are homotopy equivalent.

  \begin{rmk}\label{rmk:homotopypushout}
      If spaces are simplicial complexes and maps are subcomplex  inclusions, then their homotopy pushout and pushout spaces are equivalent up to homotopy. For elaborate discussion on this, we refer interested reader to \cite[Chapter 7]{bredon}.
  \end{rmk}

\begin{rmk} \label{rmk:pushout1}
    Consider a diagram $Y \xleftarrow{p}  X \xrightarrow{q} Z$ and let $W$ be its pushout object.  If $p,q$ are null-homotopic, then $W \simeq Y \bigvee Z \bigvee \Sigma(X)$, where $\simeq$  and  $\bigvee$ respectively denotes homotopy equivalence and the wedge of topological spaces; and $\Sigma(X)$ denotes the suspension of $X$.
\end{rmk}


\subsection{Simplicial Collapse}

Let  $\K$ be a simplicial complex and $\tau, \sigma \in \K$ be such that
$\sigma \subsetneq \tau$ and  $\tau$ is the only maximal simplex in $\K$ that contains $\sigma$.
  A  {\it simplicial collapse} of $\K$ is the simplicial complex $\mathcal{L}$ obtained from $\K$ by
  removing all those simplices $\gamma$  of $\K$ such that
  $\sigma \subseteq \gamma \subseteq \tau$. Here $\sigma$ is called a {\it free face} of
  $\tau$ and $(\sigma, \tau)$ is called a {\it collapsible pair}. We denote this collapse
  by $\K \searrow \mathcal{L}$. A complex is called {\it collapsible} if it collapses onto a point by applying a sequence of simplicial collapses. 
  It is a simple observation that, if 
  $\K \searrow \mathcal{L}$ then $\K \simeq \mathcal{L}$ (in fact, $\K$ deformation retracts onto $\mathcal{L}$).
 Throughout this article, we write $\K \searrow  \langle A_1,A_2,\dots,A_r \rangle$ to mean that $\K$ collapses onto a subcomplex whose faces are $A_1,\dots, A_r$ and all its subsets. 

We now give a result about collapsing which will be used throughout this article. To the best of our knowledge, this result is not present in writing anywhere in literature.

  \begin{lem}\label{lemma:chordalcollapsing2}
  Let $\K$ be a simplicial complex and $\sigma$ be a facet of $\K$. Let $A, B, C \subset \sigma$ be such that $A, B \neq \emptyset$ and $\sigma=A \sqcup B \sqcup C$. For each $a\in A$ and $b \in B$, let $\{a,b\}$ be a free face of $\sigma$ in $\K$. 
  \begin{enumerate}
      \item If $C \neq \emptyset$, then $\sigma \searrow \langle \sigma\setminus A, \sigma\setminus B\rangle$.
      \item If $C = \emptyset$, $|A|=1$ and $|B| \geq 2$, then for each $a\in A$ and $b \in B$, $\sigma\searrow \langle B,\{a,b\}\rangle$. 
      \item If $C = \emptyset$ and $|A|, |B| \geq 2$, then for each $a\in A$ and $b \in B$, $\sigma\searrow \langle A,B,\{a,b\} \rangle$.
  \end{enumerate}
  \end{lem}
  
  \begin{proof}
Let $A= \{a_1, \ldots, a_p\}$ and $B= \{b_1, \ldots, b_q\}$. Without loss of generality, we can assume that $a_p=a$ and $b_q=b$. 

\begin{enumerate}
\item We first do collapsing using the vertex $a_1$ from set $A$ and then use similar arguments for the rest. Since $\{a_1, b_1\}$ is a free face of $\sigma$, $\sigma \searrow \langle \sigma \setminus \{b_1\}, \sigma \setminus \{a_1\}\rangle$. Now $(\{a_1, b_2 \}, \sigma \setminus \{b_1\})$ is a collapsible pair and therefore $\sigma \setminus \{b_1\} \searrow \langle \sigma \setminus \{b_1, b_2\}, \sigma \setminus \{a_1, b_1\}\rangle$. Hence $\sigma \searrow \langle \sigma \setminus \{b_1, b_2\}, \sigma \setminus \{a_1\} \rangle$. Inductively, assume that for some $1 \leq j < q$, $\sigma \searrow  \langle\sigma\setminus \{b_1, \ldots, b_j\}, \sigma \setminus \{a_1\} \rangle$.  Now $(\{a_1, b_{j+1}\},  \sigma \setminus \{b_1, \ldots, b_j\})$ is a collapsible pair and therefore $\sigma \searrow \langle \sigma \setminus \{b_1, \ldots, b_j, b_{j+1}\}, \sigma \setminus \{a_1\}\rangle$. Using induction, we get that $\sigma \searrow  \langle \sigma\setminus \{b_1, \ldots, b_{q}\}= \sigma \setminus B, \sigma \setminus \{a_1\}\rangle$. 

If $p=1$, then this completes the proof, otherwise we proceed as follows:

By doing similar collapsing using vertices $a_2, a_3 ,\dots, a_{p-1}$ from set $A$, we get that $\sigma \searrow  \langle \sigma \setminus B, \sigma \setminus \{a_1,\dots, a_{p-1}\}\rangle$. Now $(\{a_{p}, b_1\}, \sigma \setminus \{a_1, \ldots, a_{p-1}\})$ is a collapsible pair and therefore
	$\sigma \setminus \{a_1, \ldots, a_{p-1}\} \searrow \langle \sigma \setminus \{a_1, \ldots, a_{p-1}, a_{p}\} = \sigma \setminus A, \sigma \setminus \{a_1, \ldots, a_{p-1}, b_1\}\rangle$. Thus $\sigma\searrow \langle \sigma \setminus B, \sigma \setminus A, \sigma \setminus \{a_1, \ldots, a_{p-1}, b_1\}\rangle$.
	We now show that $\sigma \setminus \{a_1, \ldots, a_{p-1}, b_1\}\searrow \langle \sigma \setminus A, \sigma \setminus B\rangle$.
	
	Since $C \neq \emptyset$,	it is easy to observe that by using collapsible pairs in the following order:
	$$(\{a_{p}, b_2\}, \sigma \setminus \{a_1, \ldots, a_{p-1}, b_1\}), \ldots, (\{a_{p}, b_{q}\}, \sigma \setminus \{a_1, \ldots, a_{p-1}, b_1, \ldots, b_{q-1}\}),$$
	
	and applying the collapses, we get  that $\sigma \setminus \{a_1, \ldots, a_{p-1}, b_1\} \searrow \langle \sigma \setminus \{a_1, \ldots, a_{p-1}, $ $a_p,  b_1, \ldots, b_{q-1}\}, \sigma \setminus \{a_1, \ldots, a_{p-1},$ $ b_1,  \ldots, b_{q-1}, b_q \} \rangle$.
	Since $\sigma \setminus \{a_1, \ldots, a_{p-1},a_p,  b_1, \ldots$ $,  b_{q-1}\} \subseteq \sigma \setminus A$ and $\sigma \setminus \{a_1, \ldots,  a_{p-1},   b_1, \ldots, b_{q-1}, b_q \} \subseteq \sigma \setminus B$, we get that $\sigma \searrow \langle \sigma \setminus A, \sigma \setminus B\rangle$.
	

	\item  Here $\sigma = \{a, b_1, \ldots, b_q\}$. Using similar arguments as in the first paragraph of previous case and using the collapsible pairs in the following order:
	$$
	(\{a, b_1\}, \sigma), (\{a, b_2\}, \{a, b_2, \ldots, b_q\}), \ldots, (\{a, b_{q-1}\}, \{a,b_{q-1}, b_q\}),
	$$ and doing the collapses, we get the desired result.

	\item Using similar arguments as in the proof of case $1$, we get that $\sigma \searrow \langle\sigma\setminus B, \sigma \setminus \{a_1,\dots,a_{p-1}\}\rangle$. Observe that $\sigma\setminus \{a_1,\dots,a_{p-1}\} = \{a_p,b_1,\dots,b_q\}$. We now use case $2$ to collapse $\sigma\setminus \{a_1,\dots,a_{p-1}\}$ and get the result.  
\end{enumerate}
  \vspace{-0.8cm}
  \end{proof}
  
  Induction along with \Cref{lemma:chordalcollapsing2} gives the following result.

\begin{cor} \label{lemma:generalcollapsing} Let $\K$ be a simplicial complex and let $\sigma = A_1 \sqcup A_2 \ldots \sqcup A_k \sqcup C $ be a facet of $\K$, where $A_j = \{a_1^j, \ldots, a_{l_j}^j\}$ for $1 \leq j \leq k$. Let $\{a_s^i, a_t^j\}$ be a free face of $\sigma$ for each $i \neq j$, $s \in [l_i]$ and $t \in [l_j]$. 
	\begin{enumerate}
	    
	    \item If $C=\emptyset$, then $\sigma \searrow \langle A_1, \ldots, A_k, \{a_{l_1}^1, a_{l_k}^k\}, \{a_{l_2}^2, a_{l_k}^k\}, \dots, \{a_{l_{k-1}}^{k-1}, a_{l_k}^k\}\rangle$.
	    
	    \item If $C\neq \emptyset$, then $\sigma \searrow \langle A_1 \sqcup C, \dots, A_k \sqcup C \rangle$.
	\end{enumerate}
	\end{cor}

\section{Structural properties of \texorpdfstring{$\dl$}{}} \label{sec:functor}

In this section, we first analyze $\dl(G)$ as the $2$-skeleton of the clique complex of $G$ and then  we use it to compute the homotopy type of $\dl(G)$ for various classes of graphs including triangle free graphs and chordal graphs. Later we realize $\dl$ as a functor from the category of graphs to the category of simplicial complexes and compute the  homotopy type of $\dl(G)$, when $G$ is a complete multipartite graph, using the functoriality of $\dl$.

 The {\it nerve} of a family of sets $(A_i)_{i \in I}$  is the simplicial complex $\mathbf{N} = \mathbf{N}(\{A_i\})$ defined on the vertex set $I$ so that a finite subset $\sigma \subseteq I$ is in $\mathbf{N}$ precisely when $\bigcap\limits_{i \in \sigma} A_i \neq \emptyset$.

\begin{thm}[Theorem 10.6 (i), \cite{bjorner}]\label{thm:nerve}
	Let $\K$ be a simplicial complex and $(\K_i)_{i \in I}$ be a family of subcomplexes such that $\K = \bigcup\limits_{i \in I} \K_i$.
Suppose every nonempty finite intersection $\K_{i_1} \cap \ldots \cap \K_{i_t}$ for $i_j \in I, t \in \mathbb{N}$ is contractible, then $\K$ and $\mathbf{N}(\{\K_i\})$ are homotopy equivalent.
		
	\end{thm}
	\begin{rmk}\label{rem:nerve}
	Since any non-empty  intersection of simplices is again a simplex, \Cref{thm:nerve} implies that the nerve of the facets of a simplicial complex $\K$  is homotopy equivalent to $\K$. 
	\end{rmk}




\begin{lem}\label{lem:2skel}
    For any graph $G$, the complex $\dl(G)$ is homotopy equivalent to the $2$-skeleton of the clique complex of $G$, {\it i.e.}, $$\dl(G)\simeq (\Delta(G))^{(2)}.$$
\end{lem}
\begin{proof}
    Let $\mathcal{F}$ be the set of facets of $\dl(G)$. It is easy to observe that any $F\in \mathcal{F}$ is either a collection of three edges of $G$ forming a triangle in $G$ or of the form $\sigma_a= \{(a, v) : v \in N_G(a)\}$ for some $a \in V(G)$. From \Cref{rem:nerve}, $\mathbf{N}(\mathcal{F})\simeq \dl(G)$. It is therefore enough to show that $\mathbf{N}(\mathcal{F})\simeq (\Delta(G))^{(2)}$.
    
    Let $\K$ be a simplicial complex obtained from $(\Delta(G))^{(2)}$ by adding a barycenter to every $2$-simplex of $(\Delta(G))^{(2)}$. Clearly $\K\simeq (\Delta(G))^{(2)}$ and $\K \cong \mathbf{N}(\mathcal{F})$.
\end{proof}

 Let $v(G)$ and $e(G)$ denote the number of vertices and edges in a graph $G$, respectively. For a given graph $G$, define $u(G)=e(G)-e(T)$, where $T$ is a spanning tree of $G$. The following is an immediate consequence of \Cref{lem:2skel}.

\begin{cor}\label{prop:nm2oftrianglefree}
  Let $G$ be a triangle free graph. Then  $\dl(G)$ is homotopy equivalent to $G$ (considered as a $1$-dimensional simplicial complex). In particular  $$\dl(G) \simeq \bigvee\limits_{u(G)}\bS^1, $$ 
\end{cor}
 A graph $G$ is called {\it chordal} if it has no induced cycle of length greater than $3$, {\it i.e.}, each cycle of length more than $3$ has a chord. 
   \begin{thm}\label{thm:chordalmain}
  Let $G$ be a connected chordal graph.  Then $\dl (G)$ is homotopy equivalent to a wedge of $2$-dimensional spheres.
  \end{thm}

\begin{proof}
    It is  well known that the $2$-skeleton of the clique complex of a chordal graph is simply connected. Also, we know that any $2$-dimensional simply connected simplicial complex is homotopy equivalent to a wedge of $2$-spheres (cf. \cite[(9.19)]{bjorner}). The result therefore follows from \Cref{lem:2skel}. 
\end{proof}
 Let $G$ be a graph. The {\it cone} over $G$ with apex vertex $w$, denoted as $C_wG$, is the graph with $V(C_wG) = V(G) \sqcup \{w\}$ and edge set $E(C_wG) = E(G) \cup \{(w, v) : v \in V(G)\}$. 
  
  
 \begin{lem}\label{lem:nm2ofconeofgraph}
  Let  $G$ be a graph with $m$ triangles. Then $\dl(C_wG)$ is homotopy equivalent to a wedge of $m$ spheres of dimension $2$.   
 \end{lem}
 
 \begin{proof} 
 
 Let $\mathcal{F}=\{F\in (\Delta(C_wG))^{(2)}: |F|=3\text{ and } w \notin F\}$. Observe that $|\mathcal{F}| = m$. Clearly  $(\Delta(C_wG))^{(2)}\setminus \mathcal{F} \cong \mathcal{C}_{w}(G)$ (here $G$ is considered as a $1$-dimensional simplicial complex) and therefore $(\Delta(C_wG))^{(2)}\setminus \mathcal{F}$ is contractible. Hence, from \cite[Proposition 7.8]{dk}, we have that $(\Delta(C_wG))^{(2)}\simeq (\Delta(C_wG))^{(2)}/((\Delta(C_wG))^{(2)}\setminus \mathcal{F})$. Thus, from \Cref{lem:2skel}, $\dl(C_wG)\simeq (\Delta(C_wG))^{(2)}\simeq \bigvee\limits_{|\mathcal{F}|}\bS^2$.
 \end{proof}
 
 
 
 The {\it suspension}  $\Sigma G$ of a graph $G$ is the graph with vertex set 
  $V(\Sigma G) = V(G) \sqcup \{a, b\}$ and $E(\Sigma G) = E(G) \cup \{(a, v) : v \in V(G)\} \cup \{(b, v) : v \in V(G)\}$. 

 \begin{lem}\label{lem:nm2ofsuspension}
    Let $G$ be a triangle free connected graph on at least two vertices. Then 
    $$
    \dl(\Sigma G) \simeq \Sigma(\dl(G)).
    $$
    \end{lem}
\begin{proof}
    We know that $\dl(\Sigma G)\simeq (\Delta(\Sigma G))^{(2)}$. Since $G$ is triangle free, $(\Delta(\Sigma G))^{(2)}= \Sigma ((\Delta(G))^{(2)})\simeq \Sigma(\dl(G)).$
\end{proof}

\subsection{Complete multipartite graphs}\label{sec:chordal}

Let $m_1, m_2, \ldots, m_r$ be positive integers and let $A_i$ be a set of cardinality $m_i$ for $1 \leq i \leq r$. A complete multipartite graph
$K_{m_1, \ldots, m_r}$ is a graph on vertex set $A_1 \sqcup \ldots \sqcup A_r$ and two vertices are adjacent if and only if they lie in different $A_i$'s. Using functoriality of $\dl$, we study the homotopy type of $\dl$ of complete multipartite graphs. The following result  can be obtained from \Cref{lem:2skel}.

\begin{cor}\label{lem:disconn graph NM2}
Let $H$ be a connected induced subgraph of connected graphs $G_1$ and $G_2$. Then,
\begin{equation*}
    \dl(G_1\bigcup\limits_{H} G_2) \simeq 
    \begin{cases}
    \dl(G_1) \bigsqcup \dl(G_2) & \rm{if~} |V(H)|=0, \\
    \dl(G_1) \bigvee \dl(G_2) & \rm{if~} |V(H)|=1, \\
    \dl(G_1) \bigcup\limits_{\dl(H)} \dl(G_2) & \rm{otherwise}.
    \end{cases}
\end{equation*}
In particular, for any $1\le s \le m_r-1$, $\dl(K_{m_1, m_2,\dots, m_r})$ is homotopy equivalent to the pushout of $$\dl(K_{m_1, m_2,\dots, m_{r-1}, s}) \hookleftarrow \dl(K_{m_1, m_2,\dots, m_{r-1}}) \hookrightarrow \dl(K_{m_1, m_2,\dots, m_{r-1},m_r-s}).$$ 
\end{cor}

\begin{prop} \label{thm:nm2ofpartitegraphs}
	For $m,n,r \in \mathbb{N}$, $\dl(K_{m,n}) \simeq \bigvee\limits_t \mathbb{S}^1$ and $\dl(K_{m,n,r}) \simeq \bigvee\limits_{t(r-1)} \mathbb{S}^2$, where $t = mn - (m+n-1)$.
\end{prop}

\begin{proof}
    For the graph $K_{m,n}$, recall $u(K_{m,n}) = e(K_{m,n}) - e(T_{m,n}) = mn - (m+n-1)$, where $T_{m,n}$ denotes a spanning tree of $K_{m,n}$.
    Since any bipartite graph is triangle-free,  \Cref{prop:nm2oftrianglefree} implies that $\dl(K_{m,n})$ is homotopy equivalent to $\bigvee\limits_{u(K_{m,n})} \mathbb{S}^1$.
    
    Further, $K_{m,n,1}$ is a cone over the triangle-free graph $K_{m,n}$, and hence $\dl(K_{m,n,1})$ is contractible by \Cref{lem:nm2ofconeofgraph}. 
    Now using \Cref{rmk:pushout1} and \Cref{lem:disconn graph NM2}, $\dl(K_{m,n,2}) \simeq \{\text{pt}\} \bigvee \{\text{pt}\} \bigvee \big{(}\bigvee\limits_{u(K_{m,n})} \mathbb{S}^2 \big{)}$.
    Inductively, we construct $K_{m,n,r}$ as a pushout of $K_{m,n,r-1} \hookleftarrow K_{m,n} \hookrightarrow K_{m,n,1}$. 
    Then repeating the similar arguments, we get that $$\dl(K_{m,n, r}) \simeq \bigvee\limits_{u(K_{m,n})(r-1)} \mathbb{S}^2.\vspace*{-0.7cm}$$
\end{proof}

\begin{note}
    The homotopy type of $\dl(K_{m,n})$ has also been computed by Linusson et al. in \cite[Theorem 1.4]{LSW} using discrete Morse theory. 
\end{note}

\begin{thm} \label{thm:multipartite}
    For a complete $r$-partite graph $K_{m_1,\dots,m_r}$, $\dl(K_{m_1,\dots,m_r})$ is homotopy equivalent to a wedge of $\mathbb{S}^2$'s for $r > 2$.
\end{thm}

\begin{proof}
    From \Cref{thm:nm2ofpartitegraphs}, $\dl(K_{m_1,m_2,m_3})$ is homotopy equivalent  to a wedge of 2-dimensional spheres. So let $r>3$. Since $K_{m_1,m_2,m_3,1}$ is a cone over $K_{m_1,m_2,m_3}$, \Cref{lem:nm2ofconeofgraph} gives that $\dl(K_{m_1,m_2,m_3,1})$ is homotopy equivalent to a wedge of $\mathbb{S}^2$'s. 
   If $m_4>1$, we inductively construct $K_{m_1,m_2,m_3,m_4}$ as a pushout of 
    $$K_{m_1,m_2,m_3,m_4-1} \hookleftarrow K_{m_1,m_2,m_3} \hookrightarrow K_{m_1,m_2,m_3,1}.$$ 
    Note that after applying $\dl$ to this diagram and using \Cref{thm:nm2ofpartitegraphs}, we get the following pushout diagram       

\begin{figure}[H]
\centering
      \begin{tikzpicture}
\node (a) at (0,0) {$\bigvee \bS^2$};
\node (c) at (0,-2) {$\big(\bigvee \bS^2\big) \vee \big{(} \bigvee \mathbb{B}^3 \big{)}$};
\node (b) at (3,0) {$\big( \bigvee \bS^2 \big) \vee \big{(} \bigvee \mathbb{B}^3 \big{)}$};
\node (d) at (3,-2) {$X$}; \node at (5.1,-2) {$\simeq \dl(K_{m_1,m_2,m_3,m_4})$};
\draw[right hook->] (a)--(b) node[midway,above right] {};
\draw[right hook->] (a)--(c) node[midway,above right] {};
\draw[->] (b) edge (d) (c) edge (d);
\end{tikzpicture}
\caption{$\dl(K_{m_1,m_2,m_3,m_4})$ as a pushout}\label{diag:pushout}
\end{figure}

Here $\mathbb{B}^3$ denotes a closed ball of dimension $3$. 
From  \Cref{diag:pushout}, it is easy to see that $\dl (K_{m_1,m_2,m_3,m_4})$ is homotopy equivalent to a wedge of spheres of dimension $2$ and $3$. 
However, \Cref{thm:leray} implies  that $\tilde{H}_3(\dl(K_{m_1,m_2,m_3,m_4}))=0$, and hence $\dl(K_{m_1,m_2,m_3,m_4}) \simeq \bigvee \mathbb{S}^2$. 
    We now repeat these arguments for $K_{m_1,\dots,m_{r}}$ to get the desired result.
\end{proof}

\section{Wheel-free graphs and 4-regular circulant graphs}\label{sec:wheel}
	
 For $n \geq 3$, a  wheel graph on $n+1$ vertices, denoted by $W_n$, is a graph isomorphic to cone over a cycle graph $C_n$. For a wheel graph, we call apex vertex of the cone as the center of the wheel. 
	
	\begin{thm}\label{thm:wheelfree}
	Let $G$ be a connected wheel-free graph. Then $\dl (G)$ is homotopy equivalent to a wedge of circles. 
\end{thm}

\begin{proof}
	In view of \Cref{lem:2skel}, we can assume that $G$ does not have any leaf. Observe that any maximal simplex $\sigma$ in $\dl (G)$ is one of the following two types
	\begin{enumerate}
		\item $\sigma=\{e_1,\dots,e_t\}$ such that $\bigcap\limits_{i\in [t]} e_i =\{x\}$, {\itshape i.e.}, every edge shares a common vertex.
		\item $\sigma=\{e_1,e_2,e_3\}$ and $e_1,e_2,e_3$ forms a triangle in $G$.
	\end{enumerate}
	We now show that each facet of $\dl (G)$ can be collapsed to a $1$-dimensional subcomplex. First we collapse facets of type $1$.\\   
	
	\noindent{\bf Case 1:} Let $\sigma = \{e_1, \dots, e_t\}$ be a facet and $\bigcap\limits_{i\in [t]} e_i =\{x\}$. Without loss of generality, we can assume that $t\geq 3$. We partition the set  $N_G(x) = A \sqcup B$, where
	
	\begin{itemize}
		\item $A= \bigcup\limits_{i \in [r_1]} \{a_i\}$, where each $a_i$ is an isolated vertex in the induced subgraph $G[N_G(x)]$.
		\item $B= \bigsqcup\limits_{j \in [r_2]} B_j$, where each $B_j \subseteq N_G(x)$ such that $G[B_j]$ is a connected component of $G[N_G(x)]$ and has more than $1$ vertex. 
	\end{itemize}
	Since $G$ is wheel-free graph, it is easy to see that $G[B_j]$ is a tree for each $j \in [r_2]$. Define,
	\begin{equation}
	\begin{split}
	\A & =\{(x,a_i): i \in [r_1]\} \mathrm{~and} \\
	\B & = \bigsqcup\limits_{j \in [r_2]} \{(x,a): a \in B_j\},
	\end{split}
	\end{equation}

	(cf. \Cref{fig:wheelfreegraph}).	For each $j \in [r_2]$, let $\B_j = \{(x, a) : a \in B_j\}$.  It is easy to see that each $\{e, e'\}$ is a free face of $\sigma $ whenever 
	\begin{itemize}
	    \item[(i)] $e \in \B_i$, $e'\in \B_j$ and $i\neq j$, or
	    \item[(ii)] $e \in \A$ and $e'\in \B$, or
	    \item[(iii)] $e, e' \in \A$.
	\end{itemize}

	For each $j \in [r_2]$, let  $\B_j = \{e_1^j, e_2^j, \ldots, e_{l_j}^j\}$ and $\A = \{e_1, \ldots, e_{r_1}\}$. By \Cref{lemma:generalcollapsing}$(1)$, $\sigma \searrow   \langle\B_1, \ldots, \B_{r_2},\A , \{e_{l_1}^1, e_{r_1}\},\dots, \{e_{l_{r_2}}^{r_2}, e_{r_1}\} \rangle$. 

 \begin{figure}[H]
    \begin{subfigure}[]{0.3\textwidth}
    				\centering
        \tikzstyle{black1}=[circle, draw, fill=black!100, inner sep=0pt, minimum width=4pt]
        \tikzstyle{ver}=[]
        \tikzstyle{extra}=[circle, draw, fill=black!00, inner sep=0pt, minimum width=2pt]
        \tikzstyle{edge} = [draw,-]
        \tikzstyle{light} = [draw, gray,-]
        \centering
        \resizebox{0.85\textwidth}{!}{%
        \begin{tikzpicture}[scale=0.6]
            \foreach \x/\y in 				{0/v_{1},40/v_{2},80/v_{3},120/v_{4},160/v_{5},200/v_{6},240/v_{7},280/v_{8}}{
            	\node[black1] (\y) at (\x:4){};
            }
            
            \foreach \x/\y in 
                { 60/w_{1}}{
            	\node[black1] (\y) at (\x:6){};
            }
            
            \node[black1] (v) at (0,0){};
            
            \foreach \x/\y in {v_{1}/v_{2},v_{2}/v_{3},v_{3}/v_{4},v_{5}/v_{6},v_{2}/w_{1}}{
            	\path[edge] (\x) -- (\y);}
            
            \foreach \x/\y in 				{v/v_{1},v/v_{2},v/v_{3},v/v_{4},v/v_{5},v/v_{6},/v_{7},v/v_{8},v/w_{1}}{
            	\path[edge] (v) -- (\y);}
            
            \node[right] at (0.1,-4.45) {$a_2$};
            \node[right] at (-2.65,-4.05) {$a_1$};
            \node[right] at (0.1,-0.45) {$x$};
            \node[below] at (4.7,0.7) {$a_1^1$};
            \node[below] at (3.8,3.5) {$a_2^1$};
            \node[right] at (0,4.6) {$a_3^1$};
            \node[above] at (2.9,5.2) {$a_4^1$};
            \node[left] at (-1.7,4.09) {$a_5^1$};
            \node[left] at (-3.7,1.4) {$a_2^1$};
            \node[left] at (-3.7,-1.3) {$a_2^2$};
    
        \end{tikzpicture}
        }%
        \caption{$G[\{x\} \cup N_G(x)]$}
    \end{subfigure}
    \begin{subfigure}[]{0.3\textwidth}
			\centering
			\tikzstyle{black1}=[circle, draw, fill=black!100, inner sep=0pt, minimum width=4pt]
            \tikzstyle{ver}=[]
            \tikzstyle{extra}=[circle, draw, fill=black!00, inner sep=0pt, minimum width=2pt]
            \tikzstyle{edge} = [draw,-]
            \tikzstyle{light} = [draw, gray,-]
            \centering
	    	\vspace{2.3 cm}	  
	    	\resizebox{0.35\textwidth}{!}{%
		\begin{tikzpicture}
				[scale=0.5, vertices/.style={draw, fill=black, circle, inner
					sep=0.5pt}]
	
	\foreach \x/\y in {240/v_{7},280/v_{8}}{
	\node[black1] (\y) at (\x:4){};
}
        \node[black1] (v) at (0,0){};
	\foreach \x/\y in {v/v_{7},v/v_{8}}{
	\path[edge] (\x) -- (\y);;
}
                \node[right] at (0.1,-0.45) {$x$};
	            \node[right] at (0.1,-4.45) {$a_2$};
                \node[right] at (-2.65,-4.05) {$a_1$};
                
				\end{tikzpicture}
				}%
				\caption{$\A$}
			\end{subfigure}
 \begin{subfigure}[]{0.3\textwidth}
			\centering
			\tikzstyle{black1}=[circle, draw, fill=black!100, inner sep=0pt, minimum width=4pt]
\tikzstyle{ver}=[]
\tikzstyle{extra}=[circle, draw, fill=black!00, inner sep=0pt, minimum width=2pt]
\tikzstyle{edge} = [draw,-]
\tikzstyle{light} = [draw, gray,-]
\centering
\resizebox{0.85\textwidth}{!}{%

\begin{tikzpicture}[scale=0.5]

\vspace{3.8cm}

\foreach \x/\y in 				{0/v_{1},40/v_{2},80/v_{3},120/v_{4},160/v_{5},200/v_{6}}{
	\node[black1] (\y) at (\x:4){};
}

\foreach \x/\y in 
    { 60/w_{1}}{
	\node[black1] (\y) at (\x:6){};
}

\node[black1] (v) at (0,0){};

\foreach \x/\y in 				{v/v_{1},v/v_{2},v/v_{3},v/v_{4},v/v_{5},v/v_{6},v/w_{1}}{
	\path[edge] (v) -- (\y);}

\node[right] at (0.1,-0.45) {$x$};
\node[below] at (2.5,0.1) {$e_1^1$};
\node[below] at (2.1,2.0) {$e_2^1$};
\node[right] at (-0.7,2.6) {$e_3^1$};
\node[left] at (2.6,4.4) {$e_4^1$};
\node[left] at (-1.3,2.3) {$e_5^1$};
\node[left] at (-2.0,0.4) {$e_2^1$};
\node[left] at (-1.1,-1.2) {$e_2^2$};
            \node[below] at (4.7,0.7) {$a_1^1$};
            \node[below] at (3.8,3.5) {$a_2^1$};
            \node[right] at (0,4.6) {$a_3^1$};
            \node[above] at (3,5.15) {$a_4^1$};
            \node[left] at (-1.7,4.09) {$a_5^1$};
            \node[left] at (-3.7,1.4) {$a_2^1$};
            \node[left] at (-3.7,-1.3) {$a_2^2$};
        \vspace{1cm}
\end{tikzpicture}
}%
\vspace{1cm}
			\caption{$\B$}
		\end{subfigure}

\vspace{0.3cm}

\begin{subfigure}[]{0.5\textwidth}
				\centering
\tikzstyle{black1}=[circle, draw, fill=black!100, inner sep=0pt, minimum width=4pt]
\tikzstyle{ver}=[]
\tikzstyle{extra}=[circle, draw, fill=black!00, inner sep=0pt, minimum width=2pt]
\tikzstyle{edge} = [draw,-]
\tikzstyle{light} = [draw, gray,-]
\centering
 
 \hspace{1.5cm}
         \resizebox{0.4\textwidth}{!}{%

\begin{tikzpicture}[scale=0.5]

\foreach \x/\y in 				{0/v_{1},40/v_{2},80/v_{3},120/v_{4}}{
	\node[black1] (\y) at (\x:4){};
}

\foreach \x/\y in 
    { 60/w_{1}}{
	\node[black1] (\y) at (\x:6){};
}

\foreach \x/\y in {v_{1}/v_{2},v_{2}/v_{3},v_{3}/v_{4},v_{2}/w_{1}}{
	\path[edge] (\x) -- (\y);}

            \node[below] at (4.7,0.7) {$a_1^1$};
            \node[below] at (3.8,3.5) {$a_2^1$};
            \node[right] at (0,4.6) {$a_3^1$};
            \node[above] at (3,5.15) {$a_4^1$};
            \node[left] at (-1.7,4.09) {$a_5^1$};
\end{tikzpicture}
}%
\caption{$G[B_1]$}
\end{subfigure}
\hspace{-2.5cm}
  \begin{subfigure}[]{0.5\textwidth}
				\centering
				\tikzstyle{black1}=[circle, draw, fill=black!100, inner sep=0pt, minimum width=4pt]
				
		\vspace{0.4cm}
		
\tikzstyle{ver}=[]
\tikzstyle{extra}=[circle, draw, fill=black!00, inner sep=0pt, minimum width=2pt]
\tikzstyle{edge} = [draw, -]
\tikzstyle{light} = [draw, gray,-]
\centering
        \resizebox{0.4\textwidth}{!}{%

	\begin{tikzpicture}[scale=0.5]

\foreach \x/\y in 				{40/v_{2},80/v_{3},120/v_{4}}{
	\node[black1] (\y) at (\x:4){};
}

\foreach \x/\y in 
    { 60/w_{1}}{
	\node[black1] (\y) at (\x:6){};
}

\foreach \x/\y in {v_{2}/v_{3},v_{3}/v_{4},v_{2}/w_{1}}{
	\path[edge] (\x) -- (\y);}
            \node[below] at (3.8,3.5) {$a_2^1$};
            \node[right] at (0,4.6) {$a_3^1$};
            \node[above] at (3,5.15) {$a_4^1$};
            \node[left] at (-1.7,4.09) {$a_5^1$};
\end{tikzpicture}
}%
    \vspace{0.5cm}
				\caption{$G[B_1-{a_1^1}]$} \label{fig:wheelfreegraphtree}
			\end{subfigure}
		\caption{}\label{fig:wheelfreegraph}
		\end{figure}

  	Again using \Cref{lemma:generalcollapsing}$(2)$, $\A \searrow  \langle \{e_1,e_{r_1}\}, \dots,\{e_{r_1 - 1},e_{r_1} \}\rangle$. We now show that each $\B_j$ can be collapsed onto a $1$-dimensional subcomplex of $\dl (G)$. We prove this by induction on the number of elements in $\B_j$. If $l_j=2 $ then the simplex $\B_j$ itself is of dimension $1$.

	Fix $j \in [r_2]$ and assume that $l_j \geq 3$. For each $1 \leq i \leq l_j$, let $e_i^j = (x, a_i^j)$.  Without loss of generality, we can assume that the $a_1^j$ is a leaf vertex in tree $G[B_j]$ and $a_2^j$ is its parent. This implies that for any $i \in \{3,\dots,l_j\}$,  there exist no triangle in $G$ which contains both the edges  $e_1^j$ and $e_i^j$.
	
	
	Since $e_1^j$ and $e_3^j$ are not part of any  triangle, $(\{e_1^j, e_3^j\}, \B_j)$ is a collapsible pair and therefore $\B_j \searrow \langle \B_j \setminus \{e_1^j\}, \B_j \setminus \{e_3^j\} \rangle$.	If $l_j = 3$, then $\B_j \setminus \{e_1^j\}$ and $\B_j \setminus \{e_3^j\}$ are $1$-dimensional. Assume $l_j \geq 4$. 
	
		Now $(\{e_1^j, e_4^j\}, \B_j \setminus \{e_3^j\})$ is a collapsible pair and therefore $ \B_j \setminus \{e_3^j\} \searrow \langle  \B_j \setminus \{e_3^j, e_4^j\}, \B_j \setminus \{e_3^j, e_1^j\} \rangle$. Since $ \B_j \setminus \{e_3^j, e_1^j\} \subseteq  \B_j \setminus \{e_1^j\}$, $\B_j \searrow \langle \B_j \setminus \{e_1^j\},  \B_j \setminus \{e_3^j, e_4^j\} \rangle$. 
		It is easy to observe that by using collapsible pairs in the following order:
		
		$$ ( \{e_1^j, e_5^j\}, \B_j \setminus \{e_3^j, e_4^j\}), \ldots ( \{e_1^j, e_{l_j}^j\},\B_j \setminus \{e_3^j, e_4^j, \ldots, e_{l_j-1}^j\})$$
			and applying the collapses, we get that  $\B_j\searrow \langle \B_j \setminus \{e_1^j\}$, $\{e_1^j, e_2^j\} \rangle $. 
	Since $a_1^j$ is a leaf vertex of $G[B_j]$, $G[B_j - \{a_1^j\}]$ is again a tree (cf \Cref{fig:wheelfreegraphtree}). Therefore, from induction, we have that $\B_j \setminus \{e_1^j\}$ collapses onto a $1$-dimensional subcomplex of $\dl(G)$. Which implies that, $\B_j$ collapses onto a subcomplex of dimension $1$.\\

	\noindent{\bf Case 2:}  Once we have collapsed all simplices of type $1$ then given any simplex $\{e_1,e_2,e_3\}$ of type $2$, it is easy to see that $(\{e_1,e_2\},\{e_1,e_2,e_3\})$ is always a collapsible pair. Thus we can collapse all these simplices to a $1$-dimensional subcomplex of $\dl(G)$.

Since $G$ is connected, $\dl(G)$ is connected. Therefore the result follows from the fact that any connected one dimensional simplicial complex is homotopy equivalent to a wedge of circles. 
\end{proof}

It is an easy observation that any graph which is not $K_4$ and has maximal degree 3 is a wheel-free graph, and hence the previous result implies the following corollary.

	\begin{cor} Let $G$ be a connected graph of maximal degree at most $3$ and $G \ncong K_4$. Then $\dl (G)$ is homotopy equivalent to a wedge of circles.
	\end{cor}
	
	
Let $n \geq 3$ be a positive integer and $S \subseteq \{ 1, 2, \ldots, \lfloor  \frac{n}{2} \rfloor \}$. 
 The {\it circulant graph} $C_n(S)$ is the graph, whose set of vertices
 $V(C_n(S)) =  \{0,1, \ldots, n-1\}$  and any two vertices $x$ and $y$
 adjacent if and only if $x-y \ (\text{mod n}) \in S \cup (-S)$, where $-S = \{n-a : a \in S\}$. Circulant graphs are also Cayley graphs of $\mathbb{Z}_n$, the cyclic group on $n$ elements. Since $n \notin S$, $C_n(S)$ is a simple graph, {\it i.e.}, does not contains any loop. Further, $C_n(S)$ are $|S \cup (-S)|$-regular graphs. We now prove a structural result for $4$-regular circulant graphs. This result will enable us to do the computation of the homotopy type of their $\dl$.
  
\begin{prop} \label{prop:circulant}
Let $G$ be a  $4$-regular circulant graph. Then each connected component of $G$ is either wheel-free or isomorphic to $K_5$ or to $\Sigma C_4$.
\end{prop}

\begin{proof}
    Let $\{s, t\}$ be the generating set of $G$ such that $s< t$ and $V(G) = \{0,1, \ldots, n-1\}$. Symmetry in the circulant graph implies that connected components of $G$ are isomorphic. Suppose $G$ is not wheel-free and say it has a subgraph $H$ isomorphic to $W_m$, a wheel on $m+1$ vertices. Since $G$ is $4$-regular, $3 \le m \leq 4$. Without loss of generality we can assume that $0$ is the center vertex of $W_m$. Clearly, $N_G(0) = \{s, t, n-s, n-t\}$.
    
    \noindent {\bf Case 1:} $m =3$.
    
    Since $W_3 \cong K_4$, we see that $|V(W_3) \cap N_G(0)| = 3$. Therefore, either $s \sim t $ or $n-s \sim n-t$ in $W_3$. In both the cases we get that $t = 2s$. Since $s < t$, $n-t < n-s$. If $N_{W_3}(0)=\{s,t,n-t\} $, then $n-t < n-s$ implies that $n-t = 3s$, thus $n= 3s + 2s = 5s$. Similar analysis  for any 3 element subset of $N_G(0)$ implies that $n =5s$. This implies that  $N_G(0) =\{s,2s,3s,4s\}$ and therefore $G[N_G(0) \cup \{0\}]$ is a clique in $G$. Moreover, $4$-regularity of $G$ implies that $G[\{0,s,2s,3s,4s\}]$ is a component.
    
    \noindent {\bf Case 2:} $m=4$.
    
        Let $a \sim b \sim c \sim d \sim a$ be the outer cycle of $W_4$. By symmetry, the vertex  $a$ is also  a center of a wheel with $5$ vertices, say $W_4'$. Let $N_G(a) = \{b,d,0,x\}$. If $x=c$, then  $\{a,b,c,0\}$ forms a $K_4$. By Case 1, we get that $\{a,b,c,d,0\}$ forms a $K_5$. Therefore, let $x \ne c$, then $x\not\sim 0$ implies that  $d \sim x \sim b\sim 0 \sim d$ is the outer cycle  of $W_4'$. Similarly, $b$ is the center of some other wheel with 5 vertices, say $W_4''$. Since $N_G(b) = \{a,0,c,x\}$, the outer cycle of $W_4''$ is given by $a \sim x \sim c \sim 0 \sim a$. Therefore, $N_G(x) = N_G(0) = \{a,b,c,d\}$. Again, the $4$-regularity of $G$ implies that $G[\{0,a,b,c,d,x\}]$ is a component and is isomorphic to $\Sigma C_4$.
    \end{proof}
    
    
 \begin{cor} \label{cor:circulant}
 Let $G$ be a  $4$-regular circulant graph. Then each connected component of $\dl (G)$ is homotopy equivalent to either a wedge of circles or a wedge of $2$-spheres. 
\end{cor}
\begin{proof}

    We first note that for any cycle graph $C_r$, $\dl(C_r) \simeq \bS^1$ whenever $r\geq 4$. From \Cref{prop:circulant},  each connected component of $G$ is either wheel-free or isomorphic to $K_5$ or to $\Sigma C_4$. Since $\dl(\text{wheel-free}) \simeq \bigvee \bS^1$ (cf. \Cref{thm:wheelfree}), $\dl (K_5) \simeq \bigvee \bS^2$ (cf. \Cref{thm:completegraph}) and $\dl (\Sigma C_4) \simeq \Sigma(\dl(C_4)) \simeq \Sigma(\bS^1) = \bS^2$ (cf. \Cref{lem:nm2ofsuspension}), the result follows.
\end{proof}

It is to note here that the computations done in this section gives the exact homotopy type of $\dl$ for all the 2, 3 and 4-regular circulant graphs. 

\section{Further directions}
    
    
  For any simplicial complex $\K \simeq \big{(}\bigvee\limits_{m}\bS^{1}\big{)}\vee \big{(}\bigvee\limits_{n}\bS^{2}\big{)}$, \Cref{lem:disconn graph NM2} shows that $\dl(\big{(}\bigvee\limits_{m}C_4\big{)}\vee \big{(}\bigvee\limits_{n}K_4\big{)}) \simeq \K$, where the wedge of graphs is taken along a vertex as $1$-dimensional simplicial complexes. It is well known that the clique complex functor $\Delta$ is universal ({\it i.e.}, given any simplicial complex $\K$, there is a graph $G$ such that $\Delta(G)\simeq \K$), whereas the line graph functor $L$ is not (for example, $K_{1,3}$ is not a line graph of any graph). This raises the following natural question. 
    
    \begin{ques}Is the functor $\dl$ universal from the category of graphs to the category of $3$-Leray simplicial complexes?
    \end{ques}
    
We note here that for all the classes of graphs considered in this article, $\dl$ has always been a wedge of equidimensional spheres. Therefore, it would be interesting to know the following. 
    
    \begin{ques}
    Can the classes of graphs be classified whose $\dl$ is a wedge of equidimensional spheres? More specifically, can we classify those graphs whose $\dl$ is simply connected?
    \end{ques}
   
\section*{Acknowledgements}
 We would like to thank Russ Woodroofe for his insightful comments and suggestions. We are also thankful to the anonymous referee for pointing out the connection between the nerve of a cover of $\dl$ and the $2$-skeleton of the clique complex, which helped us in shortening many proofs of \Cref{sec:functor} significantly. The first and third authors are partially supported by a grant from the Infosys Foundation.





  


\end{document}